\documentclass[a4paper,openbib,12pt]{article}
\usepackage{geometry}
\usepackage{amsmath}
\usepackage{amssymb}
\usepackage{amsfonts}

\newtheorem{theorem}{Theorem}[section]

\newtheorem{conjecture}[theorem]{Conjecture}
\newtheorem{corollary}[theorem]{Corollary}

\newtheorem{definition}[theorem]{Definition}

\newtheorem{lemma}[theorem]{Lemma}

\newtheorem{remark}[theorem]{Remark}

\newenvironment{proof}[1][Proof]{\noindent\textbf{#1.} }{\ \rule{0.5em}{0.5em}}
\input{tcilatex}
\begin{document}

\title{The action of matrix groups on aspherical manifolds}
\author{Shengkui Ye \\
%EndAName
Xi'an Jiaotong-Liverpool University}
\maketitle

\begin{abstract}
Let $\mathrm{SL}_{n}(\mathbb{Z})$ $(n\geq 3)$ be the special linear group
and $M^{r}$ be a closed aspherical manifold. It is proved that when $r<n,$ a
group action of $\mathrm{SL}_{n}(\mathbb{Z})$ on $M^{r}$ by homeomorphisms
is trivial if and only if the induced group homomorphism $\mathrm{SL}_{n}(%
\mathbb{Z})\rightarrow \mathrm{Out}(\pi _{1}(M))$ is trivial. For (almost)
flat manifolds, we prove a similar result in terms of holonomy groups.
Especially, when $\pi _{1}(M)$ is nilpotent, the group $\mathrm{SL}_{n}(%
\mathbb{Z})$ cannot act nontrivially on $M$ when $r<n.$ This confirms a
conjecture related to Zimmer's program for these manifolds.
\end{abstract}

\section{Introduction}

Let $\mathrm{SL}_{n}(\mathbb{Z})$ be the special linear group over integers.
The linear transformations of $\mathrm{SL}_{n}(\mathbb{Z})$ on the Euclidean
space $\mathbb{R}^{n}$ induces a natural group action on the torus $T^{n}=%
\mathbb{R}^{n}/\mathbb{Z}^{n}.$ Note that $T^{n}$ is an aspherical manifold,
i.e. the universal cover is contractible. It is believed that this action is
minimal in the following sense.

\begin{conjecture}
\label{conj}Any group action of $\mathrm{SL}_{n}(\mathbb{Z})$ $(n\geq 3)$ on
a closed aspherical $r$-manifold $M^{r}$ by homeomorphisms factors a finite
group if $r<n.$
\end{conjecture}

This conjecture is related to Zimmer's program concerning group action of
lattices in Lie groups on manifolds (see the survey articles \cite{fi,zm,sw}
for more details). A relevant conjecture is proposed by Farb and Shalen \cite%
{fb}: any smooth action of a finite-index subgroup of $\mathrm{SL}_{n}(%
\mathbb{Z})$ $(n\geq 3)$, on a compact $r$-manifold factors through a finite
group action if $r<n-1$. Compared with Farb and Shalen's conjecture,
Conjecture \ref{conj} considers topological actions and the condition is
generalized to $r<n,$ but only for aspherical manifolds. When $M=S^{1},$
Conjecture \ref{conj} is already proved by Witte \cite{wi}. Weinberger \cite%
{we} confirms the conjecture when $M$ is a torus. For $C^{1+\beta }$-group
actions of a finite-index subgroup in $\mathrm{SL}_{n}(\mathbb{Z}),$ one of
the results proved by Brown, Rodriguez-Hertz and Wang \cite{brw} confirms
Conjecture \ref{conj} for surfaces when $r<n-1$. For $C^{2}$-group actions
of cocompact lattices, Brown-Fisher-Hurtado \cite{bfh} confirms Conjecture %
\ref{conj} when $r<n-1$. Note that the $C^{0}$-actions could be very
different from smooth actions. It seems very few other cases have been
confirmed (for group actions preserving additional structures, many results
have been obtained, cf. \cite{fi,zm}).

For a group $G,$ denote by $\mathrm{Out}(G)$ the outer automorphism group.
Our first result is the following.

\begin{theorem}
\label{th1}Let $M^{r}$ be an aspherical manifold$.$ A group action of $%
\mathrm{SL}_{n}(\mathbb{Z})$ $(n\geq 3)$ on $M^{r}$ $(r\leq n-1)$ by
homeomorphisms is trivial if and only if the induced group homomorphism $%
\mathrm{SL}_{n}(\mathbb{Z})\rightarrow \mathrm{Out}(\pi _{1}(M))$ is
trivial. In particular, Conjecture \ref{conj} holds if the set of group
homomorphisms 
\begin{equation*}
\mathrm{Hom}(\mathrm{SL}_{n}(\mathbb{Z}),\mathrm{Out}(\pi _{1}(M)))=1.
\end{equation*}
\end{theorem}

An obvious application is the following.

\begin{corollary}
Any group action of $\mathrm{SL}_{n}(\mathbb{Z})$ $(n\geq 3)$ on an
aspherical manifold $M^{k}$ $(k\leq n-1)$ by homotopic-identity
homeomorphisms is trivial.
\end{corollary}

For aspherical manifolds with finitely generated nilpotent fundamental
groups (eg. Nil-manifolds), we confirm Conjecture \ref{conj} as follows.

\begin{theorem}
\label{th2}Let $M^{r}$ be an aspherical manifold$.$ If the fundamental group 
$\pi _{1}(M)$ is finitely generated nilpotent, any group action of $\mathrm{%
SL}_{n}(\mathbb{Z})$ $(n\geq 3)$ on $M^{r}$ $(r\leq n-1)$ by homeomorphisms
is trivial.
\end{theorem}

We now study group actions on (almost) flat manifolds. Recall that a closed
manifold $M$ is almost flat if for any $\varepsilon >0$ there is a
Riemannian metric $g_{\varepsilon }$ on $M$ such that $\mathrm{diam}%
(M,g_{\varepsilon })<1\ $and $g_{\varepsilon }$ is $\varepsilon $-flat.

\begin{theorem}
\label{th3}Let $M^{r}$ be a closed almost flat manifold with holonomy group $%
\Phi .$ A group action of $\mathrm{SL}_{n}(\mathbb{Z})$ $(n\geq 3)$ on $%
M^{r} $ $(r\leq n-1)$ by homeomorphisms is trivial if and only if the
induced group homomorphism $\mathrm{SL}_{n}(\mathbb{Z})\rightarrow \mathrm{%
Out}(\Phi )$ is trivial. In particular, Conjecture \ref{conj} holds if the
set of group homomorphisms 
\begin{equation*}
\mathrm{Hom}(\mathrm{SL}_{n}(\mathbb{Z}),\mathrm{Out}(\Phi ))=1.
\end{equation*}
\end{theorem}

Surprisingly, the proof of Theorem \ref{th3} will use knowledge on algebraic 
$K$-theory (Steinberg groups $\mathrm{St}_{n}(\mathbb{Z})$ and $K_{2}(%
\mathbb{Z})$, especially). The usual Zimmer's program is stated for any
lattices in high-rank semisimple Lie groups. However, Theorems \ref{th1}, %
\ref{th2}, \ref{th3} cannot hold for general lattices. For example, the
congruence subgroup $\Gamma (n,p),$ which is defined as the kernel of $%
\mathrm{SL}_{n}(\mathbb{Z})\rightarrow \mathrm{SL}_{n}(\mathbb{Z}/p)$ for a
prime $p,$ has a nontrivial finite cyclic quotient group (cf. \cite{lsz},
Theorem 1.1). The group $\Gamma (n,p)$ could act on $S^{1}$ through the
cyclic group by rotations.

In order to confirm Conjecture \ref{conj}, it's enough to show that every
group homomorphism from $\mathrm{SL}_{n}(\mathbb{Z})$ to the outer
automorphism group of the fundamental (or holonomy) group is trivial, by
Theorem \ref{th1} and Theorem \ref{th3}. Actually, Conjecture \ref{conj}
could be confirmed in this way for many other manifolds in addition to
manifolds with nilpotent fundamental groups proved in Theorem \ref{th2}.
These include the following:

\begin{itemize}
\item Flat manifolds with abelian holonomy group (see Corollary \ref{app1}
for a more general result);

\item Almost flat manifolds with dihedral, symmetric or alternating holonomy
group (cf. Lemma \ref{di});

\item Flat manifold of dimension $r\leq 5$ (cf. Corollary \ref{lflat}).
\end{itemize}

The article is organized as follows. In Section 2, we study the group action
of Steinberg group on spheres and acyclic manifolds. In Section 3, we give a
proof of Theorem \ref{th1}. Theorem \ref{th2} is proved in Section 4. In
Section 5, we study group action on flat manifolds and Theorem \ref{th3} is
proved. In the last section, we give some applications to flat manifolds
with special holonomy groups.

\section{ The action of Steinberg groups on spheres and acyclic manifolds}

\subsection{Steinberg group}

For a unitary associative ring $R$, the Steinberg group $\mathrm{St}_{n}(R)$ 
$(n\geq 3)$, is generated by $x_{ij}(r)$ for $1\leq i,j\leq n$ and $r\in R$
subject to the relations:

\begin{enumerate}
\item $x_{ij}(r_{1})\cdot x_{ij}(r_{2})=x_{ij}(r_{1}+r_{2});$

\item $[x_{ij}(r_{1}),x_{jk}(r_{2})]=x_{ik}(r_{1}r_{2});$

\item $[x_{ij}(r_{1}),x_{pq}(r_{2})]=1$ if $i\neq p$ and $j\neq q.$
\end{enumerate}

Let $R=\mathbb{Z},$ the integers. There is a natural group homomorphism $f:%
\mathrm{St}_{n}(\mathbb{Z})\rightarrow \mathrm{SL}_{n}(\mathbb{Z})$ mapping $%
x_{ij}(r)$ to the matrix $e_{ij}(r),$ which is a matrix with 1s along the
diagonal, $r$ in the $(i,j)$-th position and zeros elsewhere. Denote by $%
\omega _{ij}(-1)=x_{ij}(-1)x_{ji}(1)x_{ij}(-1)$, $h_{ij}=\omega
_{ij}(-1)\omega _{ij}(-1)$ and $a=h_{12}^{2}.$ We call $a$ a Steinberg
symbol, denoted by $\{-1,-1\}$ usually.

\begin{lemma}[Milnor \protect\cite{mi}, Theorem 10.1]
\label{mi}For $n\geq 3,$ the group $\mathrm{St}_{n}(\mathbb{Z})$ is a
central extension 
\begin{equation*}
1\rightarrow K\rightarrow \mathrm{St}_{n}(\mathbb{Z})\rightarrow \mathrm{SL}%
_{n}(\mathbb{Z})\rightarrow 1,
\end{equation*}%
where $K$ is the cyclic group of order $2$ generating by $%
a=(x_{12}(-1)x_{21}(1)x_{12}(-1))^{4}.$
\end{lemma}

\begin{lemma}
\label{quat}For distinct integers $i,j,s,t,$ we have the following.

\begin{enumerate}
\item[(i)] $[h_{ij},h_{st}]=1;$

\item[(ii)] $[h_{ij},h_{is}]=a;$

\item[(iii)] The subgroup $\langle h_{ij},h_{is}\rangle $ is isomorphic to
the quaternion group $Q_{8}.$
\end{enumerate}
\end{lemma}

\begin{proof}
(i) follows the third Steinberg relation easily. (ii) is Lemma 9.7 of Milnor 
\cite{mi} (p. 74). A direct computation shows that $h_{ij},h_{is}$ and $%
h_{ij}h_{is}$ are all elements of order $4.$ Considering (ii), $\langle
h_{ij},h_{is}\rangle $ is isomorphic to the quaternion group.
\end{proof}

Denote by $q:\mathrm{SL}_{n}(\mathbb{Z})\rightarrow \mathrm{SL}_{n}(\mathbb{Z%
}/n)$ or $\mathrm{St}_{n}(\mathbb{Z})\rightarrow \mathrm{St}_{n}(\mathbb{Z}%
/n)$ the group homomorphism induced by the ring homomorphism $\mathbb{Z}%
\rightarrow \mathbb{Z}/n,$ for some integer $n.$ Let $\mathrm{SL}_{n}(%
\mathbb{Z},n\mathbb{Z})=\ker (\mathrm{SL}_{n}(\mathbb{Z})\rightarrow \mathrm{%
SL}_{n}(\mathbb{Z}/n))$ and $\mathrm{St}_{n}(\mathbb{Z},n\mathbb{Z})=\ker (%
\mathrm{St}_{n}(\mathbb{Z})\rightarrow \mathrm{St}_{n}(\mathbb{Z}/n))$ be
the congruence subgroups$.$

\begin{lemma}
\label{no}Let $N$ be a normal subgroup of $\mathrm{St}_{n}(\mathbb{Z}).$ If $%
f(N)$ contains $\mathrm{SL}_{n}(\mathbb{Z},2\mathbb{Z}),$ then $N$ contains
the element $a.$ In particular, the normal subgroup generated by $h_{ij}$ $%
(n\geq 3)$ or $h_{ij}h_{st}$ $(n\geq 5)$ contains the element $a$ for
distinct integers $i,j,s,t.$
\end{lemma}

\begin{proof}
Recall that $e_{ij}(r)$ is a matrix with 1s along the diagonal, $r$ in the $%
(i,j)$-th position and zeros elsewhere. Since $f(N)$ contains $\mathrm{SL}%
_{n}(\mathbb{Z},2\mathbb{Z}),$ we have $x_{pq}(2)\ $or $a\cdot x_{pq}(2)\in
N $ for any integers $p\neq q\in \lbrack 1,n]$ (the interval). Note that $%
\mathrm{St}_{n}(\mathbb{Z},2\mathbb{Z})$ is normally generated by $x_{pq}(2)$
(cf. 13.18 of Magurn \cite{mg}, p.448). Therefore, $\mathrm{St}_{n}(\mathbb{Z%
},2\mathbb{Z})$ or $a\cdot $ $\mathrm{St}_{n}(\mathbb{Z},2\mathbb{Z})\subset
N$. However, it is known that the Steinberg symbol $q(a)\in \mathrm{St}_{n}(%
\mathbb{Z}/2)$ is trivial (cf. Corollary 9.9 of Milnor \cite{mi}, p.75).
Thus, $a\in a\cdot $ $\mathrm{St}_{n}(\mathbb{Z},2\mathbb{Z})=\mathrm{St}%
_{n}(\mathbb{Z},2\mathbb{Z})\subset N.$ The image $f(h_{ij})=\mathrm{diag}%
(1,\cdots ,-1,\cdots ,-1,\cdots ,1)$ (or $f(h_{ij}h_{st})$) normally
generates the congruence subgroup $\mathrm{SL}_{n}(\mathbb{Z},2\mathbb{Z})$
(cf. Ye \cite{ye}). The proof is finished.
\end{proof}

\subsection{Homology manifolds}

The generalized manifolds studied in this section are following Bredon's
book \cite{b}. Let $L=\mathbb{Z}$ or $\mathbb{Z}/p.$ All homology groups in
this section are Borel-Moore homology with compact supports and coefficients
in a sheaf $\mathcal{A}$ of modules over a principal ideal domain $L$. The
homology groups of $X$ are denoted by $H_{\ast }^{c}(X;\mathcal{A})$ and the
Alexander-Spanier cohomology groups (with coefficients in $L$ and compact
supports) are denoted by $H_{c}^{\ast }(X;L).$ We define $\dim _{L}X=\min
\{n\mid H_{c}^{n+1}(U;L)=0$ for all open $U\subset X\}.$ If $L=\mathbb{Z}/p,$
we write $\dim _{p}X$ for $\dim _{L}X.$ For integer $k\geq 0,$ let $\mathcal{%
O}_{k}$ denote the sheaf associated to the pre-sheaf $U\longmapsto
H_{k}^{c}(X,X\backslash U;L).$

\begin{definition}
An $n$-dimensional homology manifold over $L$ (denoted $n$-hm$_{L}$) is a
locally compact Hausdorff space $X$ with $\dim _{L}X<+\infty $, and $%
\mathcal{O}_{k}(X;L)=0$ for $p\neq n$ and $\mathcal{O}_{n}(X;L)$ is locally
constant with stalks isomorphic to $L$. The sheaf $\mathcal{O}_{n}$ is
called the orientation sheaf.
\end{definition}

There is a similar notion of cohomology manifold over $L$, denoted $n$-cm$%
_{L}$ (cf. \cite{b}, p.373). For a prime $p,$ denote by $\dim _{p}X$ the
cohomological dimension of $X.$

\begin{definition}
If $X$ is a compact $m$-hm$_{L}$ and $H_{\ast }^{c}(X;L)\cong H_{\ast
}^{c}(S^{m};L),$ then $X$ is called a generalized $m$-sphere over $L$.
Similarly, if $H_{0}^{c}(X;L)=L$ and $H_{k}^{c}(X;L)=0$ for $k>0,$ then $X$
is said to be $L$-acyclic.
\end{definition}

We will need the following lemmas. The first is a combination of Corollary
19.8 and Corollary 19.9 (page 144) in \cite{b} (see also Theorem 4.5 in \cite%
{bv}).

\begin{lemma}
\label{locsmith}Let $p$ be a prime and $X$ be a locally compact Hausdorff
space of finite dimension over $\mathbb{Z}_{p}.$ Suppose that $\mathbb{Z}%
_{p} $ acts on $X$ with fixed-point set $F.$

\begin{enumerate}
\item[(i)] If $H_{\ast }^{c}(X;\mathbb{Z}_{p})\cong H_{\ast }^{c}(S^{m};%
\mathbb{Z}_{p}),$ then $H_{\ast }^{c}(F;\mathbb{Z}_{p})\cong H_{\ast
}^{c}(S^{r};\mathbb{Z}_{p})$ for some $r$ with $-1\leq r\leq m.$ If $p$ is
odd, then $r-m$ is even.

\item[(ii)] If $X$ is $\mathbb{Z}_{p}$-acyclic, then $F$ is $\mathbb{Z}_{p}$%
-acyclic (in particular nonempty and connected).
\end{enumerate}
\end{lemma}

The following is a relation between dimensions of fixed point set and the
whole space (cf. Borel \cite{Bo}, Theorem 4.3, p.182).

\begin{lemma}
\label{borel}Let $G$ be an elementary $p$-group operating on a first
countable cohomology $n$-manifold $X$ mod $p.$ Let $x\in X$ be a fixed point
of $G$ on $X$ and let $n(H)$ be the cohomology dimension mod $p$ of the
component of $x$ in the fixed point set of a subgroup $H$ of $G.$ If $%
r=n(G), $ we have 
\begin{equation*}
n-r=\sum_{H}(n(H)-r)
\end{equation*}%
where $H\ $runs through the subgroups of $G$ of index $p.$
\end{lemma}

The following lemma is proved by Bredon \cite{bre2} (Theorem 7.1).

\begin{lemma}
\label{Bredon}Let $G$ be a group of order $2$ operating effectively on an $n$%
-cm over $\mathbb{Z}$, with nonempty fixed points. Let $F_{0}$ be a
connected component of the fixed point set of $G$, and $r=\dim _{2}F_{0}$.
Then $n-r$ is even (respectively odd) if and only if $G$ preserves
(respectively reverses) the local orientation around some point of $F_{0}.$
\end{lemma}

\subsection{Steinberg group acting on $\mathbb{R}^{n}$ and $S^{n}$}

\begin{lemma}
\label{bv4.11}Let $X$ be a generalized $m$-sphere over $\mathbb{Z}/2$ ( or a 
$\mathbb{Z}/2$-acyclic $m$-hm$_{\mathbb{Z}/2},$ resp.). Suppose that $\tau $
is an involution of $X$ and $F$ is a closed $\tau $-invariant submanifold.
If $F$ containing $\mathrm{Fix}(\tau )$ is an $(m-1)$-sphere (or $\mathbb{Z}%
/2$-acyclic $(m-1)$-hm$_{\mathbb{Z}/2}$, resp.), then $X\backslash F$ has
two $\mathbb{Z}/2$-acyclic components and $\tau $ interchanges them.
\end{lemma}

\begin{proof}
The proof is exactly the same as that of Lemma 4.11 of Bridson and Vogthmann 
\cite{bv}, where $F=\mathrm{Fix}(\tau ).$
\end{proof}

We now study the group action of Steinberg groups $\mathrm{St}_{n}(\mathbb{Z}%
)$ on spheres and acyclic manifolds. Compared with the proof of actions of $%
\mathrm{SL}_{n}(\mathbb{Z})$, there is no enough involutions in $\mathrm{St}%
_{n}(\mathbb{Z}).$ Note that the element $h_{ij}\in \mathrm{St}_{n}(\mathbb{Z%
})$ corresponding to the involution $\mathrm{diag}(1,\cdots ,-1,\cdots
,-1,\cdots ,1)\in \mathrm{SL}_{n}(\mathbb{Z})$ is of order $4.$ Moreover, $%
h_{ij}$ and $h_{is}$ do not commute with each other. All these facts make
the study of $\mathrm{St}_{n}(\mathbb{Z})^{\prime }$s actions difficult and
the proof for the action of $\mathrm{SL}_{n}(\mathbb{Z})$ presented in \cite%
{bv} could not be carried to study that of $\mathrm{St}_{n}(\mathbb{Z})$
easily.

\begin{theorem}
\label{th4}We have the following.

\begin{enumerate}
\item[(1)] Any group action of $\mathrm{St}_{n}(\mathbb{Z})$ $(n\geq 4)$ on
a generalized $k$-sphere $M^{k}$ over $\mathbb{Z}/2$ $(k\leq n-2)$ by
homeomorphisms is trivial.

\item[(2)] Any group action of $\mathrm{St}_{n}(\mathbb{Z})$ $(n\geq 4)$ on
a $\mathbb{Z}/2$-acyclic $k$-hm$_{\mathbb{Z}/2}$ $M^{k}$ $(k\leq n-1)$ by
homeomorphisms is trivial.
\end{enumerate}
\end{theorem}

\begin{proof}
Let $a=$ $(x_{12}(1)x_{21}(-1)x_{12}(1))^{4}$ as in Lemma \ref{mi}.

\begin{enumerate}
\item[Case 1] If $a$ acts trivially on $M^{k},$ the group action of $\mathrm{%
St}_{n}(\mathbb{Z})$ factors through an action of $\mathrm{SL}_{n}(\mathbb{Z}%
).$ However, it is proved by Bridson and Vogtmann \cite{bv} that the group
action of $\mathrm{SL}_{n}(\mathbb{Z})$ is trivial. Suppose now that $a$
acts non-trivially.

\item[Case 2] $\mathrm{Fix}(a)\neq \emptyset $ and $\mathrm{Fix}(a)\neq
M^{k}.$ Since $\mathrm{St}_{n}(\mathbb{Z})$ $(n\geq 3)$ is perfect, every
element acts by orientation-preserving homeomorphism. Bredon's result (cf.
Lemma \ref{Bredon}) shows that $\mathrm{Fix}(a)$ is of even dimension. If $%
\dim _{2}\mathrm{Fix}(a)=k,$ $\mathrm{Fix}(a)=M^{k}$ by the invariance of
domain. This is a contradiction to the fact that $a$ acts non-trivially.
Therefore, 
\begin{equation*}
\dim _{2}\mathrm{Fix}(a)\leq k-2.
\end{equation*}%
Note that $h_{12}h_{34}$ is of order $2$ and $A:=\langle
a,h_{12}h_{34}\rangle $ is isomorphic to $(\mathbb{Z}/2)^{2}.$ Write $r=\dim
_{2}(\mathrm{Fix}(A))$ and $n(H)=\dim _{2}(\mathrm{Fix}(H))$ for each
non-trivial cyclic subgroup $H<A.$ By Borel's formula (cf. Lemma \ref{borel}%
),%
\begin{equation}
k-r=\sum n(H)-r,  \tag{1}
\end{equation}%
where $H$ ranges over the nontrivial subgroup of index $2$. Since $a$ is in
the center of $\mathrm{St}_{n}(\mathbb{Z}),$ there is an group action of $%
\mathrm{SL}_{n}(\mathbb{Z})$ acts on the acyclic $\mathbb{Z}/2$-manifold or
generalized sphere $\mathrm{Fix}(a)$ induced by that of $\mathrm{St}_{n}(%
\mathbb{Z}).$ This group action is trivial by Bridson and Vogtmann \cite{bv}%
. Thus, 
\begin{equation*}
n(\langle h_{12}h_{34}\rangle )\geq r=\dim _{2}\mathrm{Fix}(a).
\end{equation*}

\item[Case 2.1] If $n(\langle h_{12}h_{34}\rangle )=r,$ we have $n(\langle
a\cdot h_{12}h_{34}\rangle )=k$ by (1). By invariance of domain, $a\cdot
h_{12}h_{34}$ acts trivially on $M^{k}.$ Take $\omega =h_{12}\omega
_{12}(-1)\omega _{34}(-1)$ and $C=\langle a,\omega \rangle .$ Note that $%
\omega ^{2}=a\cdot h_{12}h_{34}$ and $f(\omega )=f(a\cdot \omega )$ has the
form%
\begin{equation*}
\begin{pmatrix}
& 1 &  &  \\ 
-1 &  &  &  \\ 
&  &  & -1 \\ 
&  & 1 & 
\end{pmatrix}%
\in \mathrm{SL}_{4}(\mathbb{Z}).
\end{equation*}%
Therefore, $C$ acts on $M^{k}$ as a group isomorphic to $(\mathbb{Z}/2)^{2}.$
If $\dim _{2}\mathrm{Fix}(\omega )=k,$ or $\dim _{2}\mathrm{Fix}(a\cdot
\omega )=k,$ i.e. $\omega $ or $a\cdot \omega $ acts trivially on $M^{k},$
the normal subgroup in $\mathrm{St}_{n}(\mathbb{Z})$ generated by $\omega $
or $a\cdot \omega $ contains $a$ by Lemma \ref{no}. This is a contradiction
to the fact that $\mathrm{Fix}(a)\neq M^{k}.$ Therefore, we may assume that $%
\dim _{2}\mathrm{Fix}(\omega )\leq k-2$ and $\dim _{2}\mathrm{Fix}(a\cdot
\omega )\leq k-2$, considering Lemma \ref{Bredon}. According to the Borel's
formula (1), we have $k\geq 4$ and $n\geq 5.$ By the invariance of domain, $%
\mathrm{Fix}(h_{12}h_{34})=M^{k}.$ The normal subgroup in $\mathrm{St}_{n}(%
\mathbb{Z})$ generated by $h_{12}h_{34}$ contains $a$ by Lemma \ref{no}.
Thus, $\mathrm{Fix}(a)=M^{k},$ which is impossible. Similarly, $n(\langle
a\cdot h_{12}h_{34}\rangle )\neq r.$

\item[Case 2.2] If $n(\langle h_{12}h_{34}\rangle )-r\geq 2$ and $n(\langle
a\cdot h_{12}h_{34}\rangle )-r\geq 2$ (noting that $n(\langle
h_{12}h_{34}\rangle )-r$ is even), then 
\begin{equation}
k-r=2(n(\langle h_{12}h_{34}\rangle )-r)\geq 4.  \tag{2}
\end{equation}

(Note that when $n\geq 5,$ $h_{15}h_{12}h_{34}h_{15}^{-1}=a\cdot
h_{12}h_{34} $ by Lemma \ref{quat} and thus $n(\langle h_{12}h_{34}\rangle
)=n(\langle a\cdot h_{12}h_{34}\rangle )$). When $k=4,$ we have $r=0$ and $%
n(\langle h_{12}h_{34}\rangle )=2.$ Therefore, $\mathrm{Fix}(h_{12}h_{34})$
is $S^{2}$ or $\mathbb{R}^{2}$ (cf. \cite{b}, 16.32, p.388). If $\mathrm{Fix}%
(h_{12}h_{34})=\mathbb{R}^{2},$ the quaternion group $\langle
h_{12},h_{13}\rangle $ acts on $\mathrm{Fix}(h_{12}h_{34})$ with a global
fixed point in $\mathrm{Fix}(a).$ Since any finite group of
orientation-preserving homeomorphisms of the plane that fix the origin is
cyclic, $\langle h_{12},h_{13}\rangle $ cannot act effectively. A nontrivial
element in $\langle h_{12},h_{13}\rangle $ will normally generate a group
containing $a$ by Lemma \ref{no}, which is impossible.

If $\mathrm{Fix}(h_{12}h_{34})=S^{2},$ we have $n\geq 6.$ Denote by $%
B=\langle a\cdot h_{12}h_{34},a\cdot h_{34}h_{56}\rangle \cong (\mathbb{Z}%
/2)^{2}.$ Note that $\mathrm{Fix}(a)\subset \mathrm{Fix}(B).$ Write $%
r^{\prime }=\dim _{2}(\mathrm{Fix}(B))$ and $n(H)=\dim _{2}(\mathrm{Fix}(H))$
for each non-trivial cyclic subgroup $H<B.$ By Borel's formula (cf. Lemma %
\ref{borel}),%
\begin{equation}
k-r^{\prime }=\sum n(H)-r^{\prime },  \tag{3}
\end{equation}%
where $H$ ranges over the nontrivial subgroup in $B$ of index $2$. Since any
two nontrivial elements in $B$ are conjugate (cf. \cite{mg}, 12.20, p.418),
we get 
\begin{equation}
k-r^{\prime }=3(n(\langle h_{12}h_{34}\rangle )-r^{\prime })  \tag{4}
\end{equation}%
and $r^{\prime }=1$. Therefore, $\mathrm{Fix}(a)$ is a submanifold of $%
\mathrm{Fix}(B)$ of codimension 1. By Lemma \ref{bv4.11}, $a$ permutes the
two components of $\mathrm{Fix}(B)\backslash \mathrm{Fix}(a).$ However, $%
h_{12}^{2}=a$ and $h_{12}\mathrm{Fix}(B)=\mathrm{Fix}(B).$ This is
impossible.

Similar arguments using (2) and (4) prove the following. When $k=5,$ $%
r=1,n(\langle h_{12}h_{34}\rangle )=3$ and $r^{\prime }=2.$ When $k=6,$ $%
r=2, $ $n(\langle h_{12}h_{34}\rangle )=4$ and $r^{\prime }=3.$ When $k=7,$ $%
r=3,n(\langle h_{12}h_{34}\rangle )=5$ and $r^{\prime }=4.$ In all these
cases, $\mathrm{Fix}(a)$ is still a submanifold in $\mathrm{Fix}(B)$ of
codimension 1. By Lemma \ref{bv4.11}, this is impossible. When $k=8,$ we
have either $r=4,n(\langle h_{12}h_{34}\rangle )=6,r^{\prime }=5$ or $%
r=0,n(\langle h_{12}h_{34}\rangle )=4,r^{\prime }=2$. The former is
impossible for the same reason as $k=7$, while the latter is impossible for
the same reason as $k=4.$

When $k\geq 9,$ we have $n\geq 10.$ If $k-r^{\prime }\geq 6,$ then $\mathrm{%
St}_{n-6}(\mathbb{Z})$ generating by all $x_{ij}(r)$ $(7\leq i\neq j\leq n)$
acts on $\mathrm{Fix}(B).$ By Smith theory (cf. Lemma \ref{locsmith}), $%
\mathrm{Fix}(B)$ is still a generalized sphere over $\mathbb{Z}/2$ or $%
\mathbb{Z}/2$-acyclic hm$_{\mathbb{Z}/2}.$ An inductive argument shows that $%
\mathrm{St}_{n-6}(\mathbb{Z})$ acts trivially on $\mathrm{Fix}(B).$
Therefore, $\mathrm{Fix}(B)=\mathrm{Fix}(a).$ However, this is impossible
considering formulas (2) and (4). If $k-r^{\prime }\leq 5,$ we have 
\begin{equation*}
3(n(\langle h_{12}h_{34}\rangle )-r^{\prime })\leq 5.
\end{equation*}%
When $n(\langle h_{12}h_{34}\rangle )-r^{\prime }=0,$ $\mathrm{Fix}%
(B)=M^{k}. $ Then $h_{12}h_{34}$ acts trivially on $M.$ The normal subgroup
generated by $h_{12}h_{13}$ contains $a$ (cf. Lemma \ref{no}), which means $%
a $ acts trivially. This is a contradiction. When $n(\langle
h_{12}h_{34}\rangle )-r^{\prime }=1,$ we have $k-r^{\prime }=3,$ $%
k-n(\langle h_{12}h_{34}\rangle )=2$ and $k-r=4.$ Therefore, $\mathrm{Fix}%
(a) $ is a submanifold in $\mathrm{Fix}(B)$ of codimension 1, which is
impossible as above by Lemma \ref{bv4.11}.

\item[Case 3] $\mathrm{Fix}(a)=\emptyset .$ According to the Lefschetz
fixed-point theorem, this can only happen when $M$ is a generalized sphere
of odd dimension.

When $k=1,$ $M=S^{1}$ (cf. \cite{b}, 16.32, p.388). The group $\langle
h_{12},h_{13}\rangle ,$ which is isomorphic to quaternion group $Q_{8}$ (cf.
Lemma \ref{quat}), acts freely on $S^{1}.$ However, this is impossible since
any finite subgroup of $\mathrm{Homeo}_{+}(S^{1})$ is isomorphic to a
subgroup of $\mathrm{SO}(2;\mathbb{R})$ (cf. Navas \cite{na}, Prop. 1.1.1).

Assume $k=3.$ Recall that $A:=\langle a,h_{12}h_{34}\rangle $ is isomorphic
to $(\mathbb{Z}/2)^{2}.$ By Smith theory, $(\mathbb{Z}/2)^{2}$ cannot act
freely and thus $\mathrm{Fix}(h_{12}h_{34})$ is not empty. Bredon's result
(cf. Lemma \ref{Bredon}) shows that $\mathrm{Fix}(h_{12}h_{34})$ is of even
dimension. If $\dim _{2}\mathrm{Fix}(h_{12}h_{34})=3,$ $h_{12}h_{34}$ acts
trivially on $M.$ The normal subgroup in $\mathrm{St}_{n}(\mathbb{Z})$
generated by $h_{12}h_{34}$ contains $a,$ which is a contradiction to the
fact that that $\mathrm{Fix}(a)=\emptyset .$ Therefore, $\mathrm{Fix}%
(h_{12}h_{34})=S^{1}.$ Note that the quaternion group $\langle
h_{12},h_{13}\rangle $ commutes with $h_{12}h_{34}.$ Since $a$ acts freely
on $\mathrm{Fix}(h_{12}h_{34})$, so does $\langle h_{12},h_{13}\rangle ,$
which is impossible as well.

When $k=5,$ take $B=\langle a\cdot h_{12}h_{34},a\cdot h_{34}h_{56}\rangle
\cong (\mathbb{Z}/2)^{2}.$ Since $\mathrm{Fix}(h_{12}h_{34})$ is a
generalized sphere over $\mathbb{Z}/2,$ $\langle a,ah_{34}h_{56}\rangle $
can not act freely on it. Therefore, $\mathrm{Fix}(B)\neq \emptyset .$ By
(4), we have%
\begin{equation*}
k-r^{\prime }=3(n(\langle ah_{12}h_{34}\rangle )-r^{\prime }).
\end{equation*}%
If $n(\langle ah_{12}h_{34}\rangle )=r^{\prime },$ we have $k=r^{\prime }$.
Thus $ah_{12}h_{34}$ acts trivially on $M.$ The normal subgroup in $\mathrm{%
St}_{n}(\mathbb{Z})$ generated by $ah_{12}h_{34}$ contains $a,$ a
contradiction. Therefore, $r^{\prime }=2$ and $n(\langle
ah_{12}h_{34}\rangle )-r^{\prime }=1$. By Lemma \ref{bv4.11}, $a$ permutes
the two components of $\mathrm{Fix}(ah_{12}h_{34})\backslash \mathrm{Fix}%
(B), $ which is impossible by noting that $h_{12}^{2}=a.$

When $k=7,$ we may assume $n(\langle ah_{12}h_{34}\rangle )\neq r^{\prime }$
as in the proof of the case when $k=5.$ Considering formula (2), we have
either $r^{\prime }=4$, $n(\langle ah_{12}h_{34}\rangle =5$ or $r^{\prime
}=1,n(\langle ah_{12}h_{34}\rangle =3.$ For the former, apply Lemma \ref%
{bv4.11} to get a contradiction. For the latter, the quaternion group $%
\langle h_{12},h_{13}\rangle $ acts on $\mathrm{Fix}(B)=S^{1}$ freely, which
is impossible as the case of $k=3.$

Suppose that $k\geq 9.$ If $k-r^{\prime }\geq 6,$ the subgroup $\mathrm{St}%
_{n-6}(\mathbb{Z})$ generating by all $x_{ij}(r)$ $(7\leq i\neq j\leq n)$
acts trivially on $\mathrm{Fix}(B)$ by an inductive argument. This is a
contradiction to the fact that $\mathrm{Fix}(a)=\emptyset .$ If $k-r^{\prime
}\leq 5,$ $n(\langle h_{12}h_{34}\rangle )-r^{\prime }=0$ or $1$. If $%
n(\langle h_{12}h_{34}\rangle )=r^{\prime },$ $k=r^{\prime }$ and thus $%
\mathrm{Fix}(B)=M^{k}.$ Then $h_{12}h_{34}$ acts trivially on $M.$ The
normal subgroup generated by $h_{12}h_{13}$ contains $a$ (cf. Lemma \ref{no}%
), which is a contradiction to the fact $\mathrm{Fix}(a)=\emptyset $. If $%
n(\langle h_{12}h_{34}\rangle )-r^{\prime }=1,$ the element $a$ permutes $%
\mathrm{Fix}(h_{12}h_{34})\backslash \mathrm{Fix}(B)$ by Lemma \ref{bv4.11}.
This is impossible by noting that $h_{12}^{2}=a.$ The whole proof is
finished.
\end{enumerate}
\end{proof}

\begin{corollary}
\label{cor}Any group homomorphism $f:\mathrm{St}_{n}(\mathbb{Z})\rightarrow 
\mathrm{GL}_{k}(\mathbb{Z})$ $(n\geq 3,k\leq n-1)$ is trivial.
\end{corollary}

\begin{proof}
When $k=1,$ $\mathrm{GL}_{k}(\mathbb{Z})$ is abelian. Since $\mathrm{St}_{n}(%
\mathbb{Z})$ is perfect, $f$ is trivial. When $k=2,$ $f$ has its image in $%
\mathrm{SL}_{2}(\mathbb{Z}).$ Note that the projective linear group $\mathrm{%
PSL}_{2}(\mathbb{Z})=\mathbb{Z}/2\ast \mathbb{Z}/3,$ a free product. Thus $%
\mathrm{SL}_{2}(\mathbb{Z})$ does not have nontrivial perfect subgroup (cf. 
\cite{ber}, 5.8, p.48). This means that $f$ is trivial. The group $\mathrm{GL%
}_{k}(\mathbb{Z})$ acts naturally on the Euclidean space $\mathbb{R}^{k}$ by
linear transformations. When $k\geq 3,$ Theorem \ref{th4} implies that the
image $\func{Im}f$ acts trivially on $\mathbb{R}^{k}.$ Therefore, $\func{Im}%
f=1.$
\end{proof}

\bigskip

\section{Proof of Theorem \protect\ref{th1}}

We need the following lemma.

\begin{lemma}
\label{cs}Denote by $Q$ a quotient group of $\mathrm{SL}_{n}(\mathbb{Z}).$
Let $\pi $ be a torsion-free abelian group. For any $n\geq 3$, the second
cohomology group 
\begin{equation*}
H^{2}(Q;\pi )=0.
\end{equation*}
\end{lemma}

\begin{proof}
By van der Kallen \cite{van}, the second homology group $H_{2}(\mathrm{SL}%
_{n}(\mathbb{Z});\mathbb{Z})=\mathbb{Z}/2$ when $n\geq 5$ and 
\begin{equation*}
H_{2}(\mathrm{SL}_{3}(\mathbb{Z});\mathbb{Z})=H_{2}(\mathrm{SL}_{4}(\mathbb{Z%
});\mathbb{Z})=\mathbb{Z}/2\tbigoplus \mathbb{Z}/2.
\end{equation*}%
Since $\mathrm{SL}_{n}(\mathbb{Z})$ is perfect, $H_{1}(\mathrm{SL}_{n}(%
\mathbb{Z});\mathbb{Z})=0$ for any $n\geq 3.$ By the universal coefficient
theorem, $H^{2}(\mathrm{SL}_{n}(\mathbb{Z});\pi )=0$ for any $n\geq 3.$
Dennis and Stein proved that%
\begin{equation*}
H_{2}(\mathrm{SL}_{n}(\mathbb{Z}/k);\mathbb{Z})=\mathbb{Z}/2,\text{ for }%
k\equiv 0(\mathrm{mod}4),
\end{equation*}%
while $H_{2}(\mathrm{SL}_{n}(\mathbb{Z}/k);\mathbb{Z})=0$, otherwise (cf. 
\cite{ds} corollary 10.2). By the universal coefficient theorem again, $%
H^{2}(\mathrm{SL}_{n}(\mathbb{Z}/k);\pi )=0$ for any $k.$ Let $f:\mathrm{SL}%
_{n}(\mathbb{Z})\rightarrow Q$ be a surjective homomorphism. By Margulis'
normal subgroup theorem, every quotient $Q$ is either $\mathrm{PSL}_{n}(%
\mathbb{Z})$ or a finite group. If $\ker f$ is trivial, $Q=\mathrm{SL}_{n}(%
\mathbb{Z})$ and thus $H^{2}(Q;\pi )=0.$ If $\ker f$ is nontrivial, the
congruence subgroup property \cite{bsm} implies that $Q$ is a quotient of $%
\mathrm{SL}_{n}(\mathbb{Z}/k)$ by a central subgroup $K$ for some non-zero
integer $k.$ From the Serre spectral sequence 
\begin{equation*}
H^{p}(Q;H^{q}(K;\pi ))\Longrightarrow H^{p+q}(\mathrm{SL}_{n+1}(\mathbb{Z}%
/k);\pi ),
\end{equation*}%
we have the exact sequence%
\begin{eqnarray*}
0 &\rightarrow &H^{1}(Q;\pi )\rightarrow H^{1}(\mathrm{SL}_{n}(\mathbb{Z}%
/k);\pi )\rightarrow H^{0}(Q;H^{1}(K;\pi )) \\
&\rightarrow &H^{2}(Q;\pi )\rightarrow H^{2}(\mathrm{SL}_{n}(\mathbb{Z}%
/k);\pi ).
\end{eqnarray*}%
This implies $H^{2}(Q;\pi )=0.$
\end{proof}

\bigskip

\begin{proof}[Proof of Theorem \protect\ref{th1}]
If $\mathrm{SL}_{n}(\mathbb{Z})$ acts trivially on $M^{r},$ it is obvious
that the induced homomorphism $\mathrm{SL}_{n}(\mathbb{Z})\rightarrow 
\mathrm{Out}(\pi _{1}(M))$ is trivial. It is enough to prove the other
direction. Denote by $\mathrm{Homeo}(M^{r})$ the group of homeomorphisms of $%
M^{r}.$ Suppose that $f:\mathrm{SL}_{n}(\mathbb{Z})\rightarrow \mathrm{Homeo}%
(M^{k})$ is the group homomorphism. We have a group extension%
\begin{equation}
1\rightarrow \pi _{1}(M)\rightarrow G^{\prime }\rightarrow \func{Im}%
f\rightarrow 1,  \tag{*}
\end{equation}%
where $\tilde{M}$ is the universal cover of $M$ and $G^{\prime }$ is a
subgroup of $\mathrm{Homeo}(\tilde{M}).$ Note that there is a one-to-one
correspondence between the equivalence classes of extensions and the second
cohomology group $H^{2}(\func{Im}f;C(\pi ))$, where $C(\pi )$ is the center
of $\pi _{1}(M)$ (see \cite{Br}, Theorem 6.6, p.105). By the assumption that
the group homomorphism $\mathrm{SL}_{n}(\mathbb{Z})\rightarrow \mathrm{Out}%
(\pi _{1}(M))$ is trivial, $\func{Im}f$ acts trivially on the center $C(\pi
).$ Since $M$ is aspherical, $\pi =\pi _{1}(M)$ is torsion-free and the
center $C(\pi )$ is torsionfree as well. By Lemma \ref{cs}, $H^{2}(\func{Im}%
f;C(\pi ))=0,$ which implies that the exact sequence (*) is split.
Therefore, $\func{Im}f$ is isomorphic to a subgroup of $G^{\prime },$ which
implies that the group $\mathrm{SL}_{n}(\mathbb{Z})$ and thus $\mathrm{St}%
_{n}(\mathbb{Z})$ could act on the acyclic manifold $\tilde{M}$ through $%
\func{Im}f.$ However, Bridson and Vogtmann \cite{bv} prove that any action $%
\mathrm{SL}_{n}(\mathbb{Z})$ on $\tilde{M}$ is trivial. (for
self-containedness, we may apply Theorem \ref{th4} to get that any group
action of $\mathrm{St}_{n}(\mathbb{Z})$ $(n\geq 4)$ on the acyclic manifold $%
\tilde{M}$ is trivial. When $n=3,$ for each integer $2\leq i\leq n,$ denote
by $A_{i}$ the diagonal matrix $\mathrm{diag}(-1,\cdots ,-1,\cdots ,1)$,
where the second $-1$ is the $i$-th entry$.$ The subgroup $G:=\langle
A_{2},A_{3}\rangle <\mathrm{SL}_{3}(\mathbb{Z})$ is isomorphic to $(\mathbb{Z%
}/2)^{2}.$ By Smith theory (cf. Lemma \ref{locsmith}), the group action of $%
G $ on $\tilde{M}$ has a global fixed point. The Borel formula (cf. Lemma %
\ref{borel}) implies that the action of $G$ on $\tilde{M}$ is trivial.
Therefore, the group action of $\mathrm{SL}_{3}(\mathbb{Z})$ on $\tilde{M}$
factors through the projective linear group $\mathrm{PSL}_{3}(\mathbb{Z}/2).$
Using Smith theory and the Borel formula once again for the subgroup $(%
\mathbb{Z}/2)^{2}\cong \langle e_{12}(1),e_{13}(1)\rangle <\mathrm{PSL}_{3}(%
\mathbb{Z}/2),$ we see that the group action of $\mathrm{PSL}_{3}(\mathbb{Z}%
/2)$ and thus $\mathrm{SL}_{3}(\mathbb{Z})$ on $\tilde{M}$ is trivial.) This
implies $\func{Im}f$ is trivial, i.e. the group action of $\mathrm{SL}_{n}(%
\mathbb{Z})$ on $M$ is trivial.
\end{proof}

\section{Aspherical manifolds with nilpotent fundamental groups}

Let $G$ be a group. Denote by $Z_{1}=Z(G),$ the center. Inductively, define $%
Z_{i+1}(G)=p_{i}^{-1}Z(G/Z_{i-1}(G)),$ where $p_{i}:G\rightarrow G/Z_{i}(G)$
is the quotient group homomorphism. We have the upper central sequence%
\begin{equation*}
1\subset Z_{1}\subset Z_{2}\subset ...\subset Z_{i}\subset ....
\end{equation*}%
If $Z_{i}=G$ for some $i,$ we call $G$ a nilpotent group. For two groups $G$
and $H,$ denote by $\mathrm{Hom}(G,H)$ the set of group homomorphisms from $%
G $ to $H.$

\begin{lemma}
\label{berrick}Let 
\begin{equation*}
1\rightarrow N\rightarrow \pi \overset{q}{\rightarrow }Q\rightarrow 1
\end{equation*}%
be a central extension, i.e. an exact sequence with $N<Z(\pi ).$ Suppose
that $G$ is a group with the second cohomology group $H^{2}(G;N)=0,$ where $%
G $ acts on $N$ trivially. Then%
\begin{equation*}
\mathrm{Hom}(G,\pi )\overset{q_{\ast }}{\rightarrow }\mathrm{Hom}(G,Q)
\end{equation*}%
is surjective.
\end{lemma}

\begin{proof}
This is an easy exercise in homological algebra. For completeness, we give a
proof here. The central extension gives a principal fibration $\mathrm{B}%
N\rightarrow \mathrm{B}\pi \rightarrow \mathrm{B}Q$ and thus a fibration 
\begin{equation*}
\mathrm{B}\pi \rightarrow \mathrm{B}Q\overset{h}{\rightarrow }K(N,2),
\end{equation*}%
where $\mathrm{B}(-)$ is a classifying space and $K(N,2)$ is a simply
connected CW complex with the second homotopy group $N$ and all other
homotopy groups trivial (cf. \cite{ber}, 8.2, p.64). Let $\alpha
:G\rightarrow Q$ be any group homomorphism. The composite 
\begin{equation*}
\mathrm{B}G\overset{\mathrm{B}\alpha }{\rightarrow }\mathrm{B}Q\overset{h}{%
\rightarrow }K(N,2)
\end{equation*}%
is null-homotopic, by the assumption that $H^{2}(G;N)=0$. Therefore, $\alpha 
$ could be lifted to be a group homomorphism $\alpha ^{\prime }:G\rightarrow
\pi .$
\end{proof}

\begin{lemma}
\label{ni}Let $\pi $ be a group with center $Z=Z(\pi ).$ Suppose that one of
the following holds:

\begin{enumerate}
\item[(i)] $G$ is a perfect group with $H_{2}(G;\mathbb{Z})$ finite, $\pi $
and $\pi /Z$ are torsion-free; or

\item[(ii)] $G$ is a perfect group with $H_{2}(G;\mathbb{Z})=0$.

If the set of group homomorphisms 
\begin{equation*}
\mathrm{Hom}(G,\mathrm{Aut}(Z))=1\text{ and }\mathrm{Hom}(G,\mathrm{Out}(\pi
/Z))=1,
\end{equation*}
then 
\begin{equation*}
\mathrm{Hom}(G,\mathrm{Out}(\pi ))=1.
\end{equation*}%
Here $1$ denote the trivial group homomorphism.
\end{enumerate}
\end{lemma}

\begin{proof}
Considering the quotient group homomorphism $\pi \rightarrow \pi /Z$, we
have the following commutative diagram%
\begin{equation*}
\begin{array}{ccccc}
1\rightarrow & \mathrm{Inn}(\pi )\rightarrow & \mathrm{Aut}(\pi )\rightarrow
& \mathrm{Out}(\pi ) & \rightarrow 1 \\ 
& \downarrow & \downarrow f & \downarrow g &  \\ 
1\rightarrow & \mathrm{Inn}(\pi /Z)\rightarrow & \mathrm{Aut}(\pi
/Z)\rightarrow & \mathrm{Out}(\pi /Z) & \rightarrow 1.%
\end{array}%
\end{equation*}%
Note that $\mathrm{Inn}(\pi )=\pi /Z.$ By the snake lemma for groups (cf. 
\cite{mg}, 11.8, p.363), the following sequence is exact%
\begin{equation}
1\rightarrow Z(\pi /Z)\rightarrow \ker f\rightarrow \ker g\rightarrow 1. 
\tag{5}
\end{equation}%
By diagram chase, the action of $\ker g$ on the center $Z(\pi /Z)$ is by
inner automorphisms of $\pi /Z$ and thus trivial. This proves that the
previous exact sequence (5) is a central extension. Since 
\begin{equation*}
\mathrm{Hom}(G,\mathrm{Out}(\pi /Z))=1
\end{equation*}%
by assumption, it suffices to prove $\mathrm{Hom}(G,\ker g)=1.$ Let $\alpha
:G\rightarrow \ker g$ be any group homomorphism. In the case of (i), when $%
\pi /Z$ is torsion-free, the center $Z(G/Z)$ is torsion-free. Since $G$ is
perfect and $H_{2}(G;\mathbb{Z})$ is finite, we have $H^{2}(G;Z(\pi /Z))=0$
by the universal coefficient theorem. In the case of (ii), we also have $%
H^{2}(G;Z(\pi /Z))=0$ using the universal coefficient theorem. Lemma \ref%
{berrick} implies that $\alpha $ could be lifted to be a group homomorphism $%
\alpha ^{\prime }:G\rightarrow \ker f.$

Let $F:\mathrm{Aut}(\pi )\rightarrow \mathrm{Aut}(Z)$ be the restriction of
automorphisms of $\pi $ to that of the center $Z.$ Since the image $F(\ker
f)\ $is a subgroup of $\mathrm{Aut}(Z)$ and 
\begin{equation*}
\mathrm{Hom}(G,\mathrm{Aut}(Z))=1
\end{equation*}%
by assumption, the map $\alpha ^{\prime }$ has its image in $\ker F\cap \ker
f.$ It is well-know that $(\ker F\cap \ker f)$ is isomorphic to $H^{1}(\pi
/Z;Z)$ (cf. \cite{gu}, Prop. 5, p.45). Since $G$ is perfect, $\alpha
^{\prime }$ has a trivial image. This proves that $\alpha $ is trivial and
thus $\mathrm{Hom}(G,\mathrm{Out}(\pi ))=1.$
\end{proof}

\bigskip

Recall that a group $G$ has cohomological dimension $k$ (denoted as $\mathrm{%
cd}(G)=k$) if the cohomological group $H^{i}(G;A)=0$ for any $i>k$ and $%
\mathbb{Z}G$-module $A,$ but $H^{n}(G;M)\neq 0$ for some $\mathbb{Z}G$%
-module $M.$ The Hirsch number or Hirsch length of a polycyclic group $G$ is
the number of infinite factors in its subnormal series. The following lemma
is well-known (see Gruenberg \cite{gu}, p.152).

\begin{lemma}
If $G$ is finitely generated, torsion-free nilpotent group, then $\mathrm{cd}%
(G)=h(G),$ where $\mathrm{cd}(G)$ is the cohomological dimension and $h(G)$
is the Hirsch number.
\end{lemma}

\begin{lemma}
Let $1\rightarrow Z\rightarrow G\rightarrow Q\rightarrow $ be a central
extension with $Z=Z(G)$ the center and $G$ a torsion-free nilpotent group.
Then $Q$ is torsion-free.
\end{lemma}

\begin{proof}
It's known that all the quotient $Z_{i}/Z_{i-1}$ is torsion-free. Suppose
that the nilpotency class of $G$ is $n,$ i.e. $Z_{n}=G.$ Then $G/Z_{n-1}$ is
(torsion-free) abelian and we have an exact sequence 
\begin{equation*}
1\rightarrow Z_{n-1}/Z_{n-2}\rightarrow G/Z_{n-2}\rightarrow
G/Z_{n-1}\rightarrow 1.
\end{equation*}%
Since both $G/Z_{n-1}$ and $Z_{n-1}/Z_{n-2}$ are both finitely generated
(Note that every subgroup of a finitely generated nilpotent group is
finitely generated) and of finite cohomology dimensions, $G/Z_{n-2}$ is of
finite cohomological dimension and thus torsion-free. Inductively, we prove
the lemma.
\end{proof}

\begin{lemma}
\label{hom}Let $G$ be a finitely generated torsion-free nilpotent group of
cohomological dimension $k.$ When $k<n,$ the set of group homomorphisms 
\begin{equation*}
\mathrm{Hom}(\mathrm{St}_{n}(\mathbb{Z}),\mathrm{Out}(G))=1,
\end{equation*}%
and thus $\mathrm{Hom}(\mathrm{SL}_{n}(\mathbb{Z}),\mathrm{Out}(G))=1.$
\end{lemma}

\begin{proof}
Let $Z$ be the center of $G.$ Note that cohomological dimension $\mathrm{cd}%
(G)=h(G)$ and $\mathrm{cd}(G/Z)=h(G/Z)$ and $h(G/Z)\leq h(G)-1.$ When $G/Z$
is abelian, we have $h(G/Z)\leq k$ and $\mathrm{Out}(G/Z)=\mathrm{GL}%
_{h(G/Z)}(\mathbb{Z}).$ Noting that 
\begin{equation*}
\mathrm{Hom}(\mathrm{St}_{n}(\mathbb{Z}),\mathrm{GL}_{k}(\mathbb{Z}))=1
\end{equation*}%
for any $k\leq n-1$ by Lemma \ref{cor}, we have $\mathrm{Hom}(\mathrm{St}%
_{n}(\mathbb{Z}),\mathrm{Out}(G/Z))=1.$ Similarly, we get $\mathrm{Hom}(%
\mathrm{SL}_{n}(\mathbb{Z}),\mathrm{Aut}(Z))=1$ by noting that $Z$ is
torsion-free abelian and $h(Z)\leq k.$ Using Lemma \ref{ni} repeatedly, we
have $\mathrm{Hom}(\mathrm{St}_{n}(\mathbb{Z}),\mathrm{Out}(G))=1.$
\end{proof}

\bigskip

Let $M$ be aspherical manifolds with a finitely generated nilpotent
fundamental group $\pi _{1}(M).$ Any group action of $\mathrm{SL}_{n}(%
\mathbb{Z})$ $(n\geq 3)$ on $M^{k}$ $(k<n)$ is trivial, as following.

\begin{proof}[Proof of Theorem \protect\ref{th2}]
Since $M$ is aspherical, $M$ is a classifying space for $\pi _{1}(M)$ and
thus the cohomological dimension $\mathrm{cd}(\pi _{1}(M))\leq r.$ By Lemma %
\ref{hom}, any group homomorphism $\mathrm{SL}_{n}(\mathbb{Z})\rightarrow 
\mathrm{Out}(\pi _{1}(M))$ is trivial. By Theorem \ref{th1}, any group
action of $\mathrm{SL}_{n}(\mathbb{Z})$ on $M$ is trivial.
\end{proof}

\section{Flat and almost flat manifolds}

A closed manifold $M$ is almost flat if for any $\varepsilon >0$ there is a
Riemannian metric $g_{\varepsilon }$ on $M$ such that $\mathrm{diam}%
(M,g_{\varepsilon })<1\ $and $g_{\varepsilon }$ is $\varepsilon $-flat. Let $%
G$ be a simply connected nilpotent Lie group. Choose a maximal compact
subgroup $C$ of $\mathrm{Aut}(G).$ If $\pi $ is a torsion-free uniform
discrete subgroup of the semi-direct product $G\rtimes C,$ the orbit space $%
M=\pi \backslash G$ is called an infra-nilmanifold and $\pi $ is called a
generalized Bieberbach group. The group $F:=\pi /(\pi \cap G)$ is called the
holonomy group of $M.$ When $G=\mathbb{R}^{n},$ the abelian Lie group, $M$
is called a flat manifold. By Gromov and Ruh \cite{gr,ru}, every almost flat
is diffeomorphic to an infra-nilmanifold. Note that $N:=\pi \cap G$ is the
unique maximal nilpotent normal subgroup of $\pi $.

The automorphism of $\pi $ is studied by Igodt and Malfait \cite{im},
generalizing the corresponding result for flat manifolds obtained by Charlap
and Vasquez \cite{cv}. Let's recall the relevant results as follows. Let $%
\psi :F\rightarrow \mathrm{Out}(N)$ be an injective group homomorphism and 
\begin{equation*}
1\rightarrow N\rightarrow \pi \rightarrow F\rightarrow 1
\end{equation*}%
be a group extension compatible with $\psi .$ The extension determines a
cohomology class $a\in H^{2}(F;N)$ (the set of 2-cohomoIogy classes
compatible with $\psi $). Let $p:$ $\mathrm{Aut}(N)\rightarrow \mathrm{Out}%
(N)$ be the natural quotient map and $\mathcal{M}_{\psi }=p^{-1}(N_{\mathrm{%
Out}(N)}(F)),$ the preimage of the normalizer of $\psi (F)$ in $\mathrm{Out}%
(N)$. Denote by $\mathcal{M}_{\psi ,a}$ the stabilizer of $a$ under the
action of $\mathcal{M}_{\psi }$ on $H^{2}(F;N)$ by conjugations. Let $A:%
\mathrm{Aut}(\pi )\rightarrow \mathrm{Aut}(N)$ be the group homomorphisms of
restrictions. For an element $n\in N,$ write $\mu (n)\in \mathrm{Aut}(N)$
the inner automorphism determined by $n,$ i.e. $\mu (n)(x)=nxn^{-1}.$ Denote
by $\ast :\mathcal{M}_{\psi }\rightarrow \mathrm{Aut}(F)$ the group
homomorphism given by $\ast (v)=\psi \circ \mu (p(v))\circ \psi ^{-1}$ (cf. 
\cite{im}, Lemma 3.3).

We will need the following result (cf. \cite{im}, Theorem 4.6, Theorem 4.8
and its proof).

\begin{lemma}
\label{key}The following sequences are exact:%
\begin{equation*}
1\rightarrow Z^{1}(F;Z(N))\rightarrow \mathrm{Aut}(\pi )\overset{A}{%
\rightarrow }\mathcal{M}_{\psi ,a}\rightarrow 1
\end{equation*}%
and 
\begin{equation*}
1\rightarrow H^{1}(F;Z(N))\rightarrow \mathrm{Out}(\pi )\overset{A}{%
\rightarrow }Q\rightarrow 1.
\end{equation*}%
The quotient $Q=(\mathcal{M}_{\psi ,a}/\mathrm{Inn}(N))/F$ and fits into an
exact sequence%
\begin{eqnarray}
1 &\rightarrow &Q_{2}=((\ker (\ast )\cap \mathcal{M}_{\psi ,a})/\mathrm{Inn}%
(N))/Z(F)\rightarrow Q  \notag \\
&\rightarrow &Q_{1}=\func{Im}(\ast |_{\mathcal{M}_{\psi ,a}})/\mathrm{Inn}%
(F)\rightarrow 1,  \TCItag{6}
\end{eqnarray}%
where $(\ker (\ast )\cap \mathcal{M}_{\psi ,a})/\mathrm{Inn}(N)$ is
contained in the centralizer $C_{\mathrm{Out}(N)}(F).$
\end{lemma}

\begin{proof}[Proof of Theorem \protect\ref{th3}]
\bigskip If $\mathrm{SL}_{n}(\mathbb{Z})$ acts trivially on $M^{r},$ it is
obvious that the induced homomorphism $\mathrm{SL}_{n}(\mathbb{Z}%
)\rightarrow \mathrm{Out}(F)$ is trivial. In order to prove the converse,
it's enough to prove $\mathrm{SL}_{n}(\mathbb{Z})\rightarrow \mathrm{Out}%
(\pi )$ is trivial for a given group action of $\mathrm{SL}_{n}(\mathbb{Z})$
on $M,$ considering Theorem \ref{th1}. We will use the exact sequence (6) in
Lemma \ref{key}. Note that $Q_{1}$ is a subgroup of $\mathrm{Out}(F).$ By
the assumption that the group homomorphism $\mathrm{SL}_{n}(\mathbb{Z}%
)\rightarrow \mathrm{Out}(F)$ is trivial, the composite 
\begin{equation*}
\mathrm{SL}_{n}(\mathbb{Z})\rightarrow \mathrm{Out}(\pi )\rightarrow
Q\rightarrow Q_{1}
\end{equation*}%
has to be trivial. Therefore, the map $\mathrm{SL}_{n}(\mathbb{Z}%
)\rightarrow \mathrm{Out}(\pi )\rightarrow Q$ has its image in $Q_{2}.$
Denote by $K=(\ker (\ast )\cap \mathcal{M}_{\psi ,a})/\mathrm{Inn}(N)$ and $%
Z=Z(F)$ to fit into an exact sequence%
\begin{equation}
1\rightarrow Z\rightarrow K\rightarrow Q_{2}\rightarrow 1.  \tag{7}
\end{equation}%
Since $K$ is a subgroup of $C_{\mathrm{Out}(N)}(F)$ by Lemma \ref{key}, the
center $Z(F)$ lies in the center of $K.$ Therefore, the exact sequence (7)
is a central extension. Let $\mathrm{St}_{n}(\mathbb{Z})$ be the Steinberg
group and denote by the composite 
\begin{equation*}
\alpha :\mathrm{St}_{n}(\mathbb{Z})\rightarrow \mathrm{SL}_{n}(\mathbb{Z}%
)\rightarrow \mathrm{Out}(\pi )\rightarrow Q_{2}.
\end{equation*}%
Since $H_{2}(\mathrm{St}_{n}(\mathbb{Z});\mathbb{Z})=0,$ we have $H^{2}(%
\mathrm{St}_{n}(\mathbb{Z});Z)=0$ by the universal coefficient theorem.
Therefore, $\alpha $ could be lifted to be a group homomorphism $\alpha
^{\prime }:\mathrm{St}_{n}(\mathbb{Z})\rightarrow K$ by Lemma \ref{berrick}.
Note that $K$ is a subgroup of $\mathrm{Out}(N)$ and the cohomological
dimension of $N\ $is at most $r.$ By Lemma \ref{hom}, $\alpha ^{\prime }$ is
trivial and thus $\alpha $ is trivial. This implies the group homomorphism $%
\mathrm{SL}_{n}(\mathbb{Z})\rightarrow \mathrm{Out}(\pi )$ has its image in $%
H^{1}(F;Z(N)).$ Since $H^{1}(F;Z(N))$ is abelian and $\mathrm{SL}_{n}(%
\mathbb{Z})$ is perfect, the map $\mathrm{SL}_{n}(\mathbb{Z})\rightarrow 
\mathrm{Out}(\pi )$ has to be trivial. The proof is finished.
\end{proof}

\begin{remark}
In the proof of Theorem \ref{th3}, we use an essential property of the
Steinberg group that $H^{2}(\mathrm{St}_{n}(\mathbb{Z});A)=0$ for any
abelian group $A.$ This does not hold for $\mathrm{SL}_{n}(\mathbb{Z})$ ($%
n\geq 3$) since $H_{2}(\mathrm{SL}_{n}(\mathbb{Z});\mathbb{Z})=\mathbb{Z}/2.$
If the abelian group $A$ is torsion-free (as in the proofs of Theorems \ref%
{th1},\ref{th2}), we still have $H^{2}(\mathrm{SL}_{n}(\mathbb{Z});A)=0$ and
it is thus not essential to use $\mathrm{St}_{n}(\mathbb{Z}).$
\end{remark}

\section{Examples}

In this section, we give further applications of Theorem \ref{th1} and
Theorem \ref{th3}.

The proof of the following lemma is similar to that of the corresponding
result for $\mathrm{SL}_{n}(\mathbb{Z}),$ proved by Kielak \cite{ki}
(Theorem 2.27).

\begin{lemma}
\label{trg}Let $p$ be a prime. Then $\mathrm{Hom}(\mathrm{St}_{n}(\mathbb{Z}%
),\mathrm{GL}_{k}(\mathbb{Z}/p))=1,$ i.e. any group homomorphism $g:\mathrm{%
St}_{n}(\mathbb{Z})\rightarrow \mathrm{GL}_{k}(\mathbb{Z}/p)$ is trivial,
for $k<n-1.$
\end{lemma}

\begin{proof}
Let $N=\ker g.$ Since the image of $N$ in $\mathrm{SL}_{n}(\mathbb{Z})$ is
of finite index, the map $g$ factors through $g^{\prime }:\mathrm{St}_{n}(%
\mathbb{Z}/N)\rightarrow \mathrm{GL}_{k}(\mathbb{Z}/p)$ for some integer $N.$
Note that $\mathrm{St}_{n}(R_{1}\times R_{2})=\mathrm{St}_{n}(R_{1})\times 
\mathrm{St}_{n}(R_{2})$ for rings $R_{1}$ and $R_{2}.$ Without loss of
generality, we assume that $N$ is a power of a prime number. Let $Z$ be the
center of $\mathrm{St}_{n}(\mathbb{Z}/N).$ Suppose that $\mathrm{GL}_{k}(%
\mathbb{Z}/p)$ acts on $(\mathbb{Z}/p)^{k}$ naturally. We could assume that
the action of $\func{Im}g^{\prime }$ on $(\mathbb{Z}/p)^{k}$ is irreducible.
Note that $(\mathbb{Z}/p)^{k}$ is the intersection of eigenspaces of $%
g^{\prime }(v),v\in Z$ (if necessary, we may consider the algebraic closure
of $\mathbb{Z}/p$). After change of basis in $(\mathbb{Z}/p)^{k},$ we get
that $g^{\prime }(N)$ lies in the center of $\mathrm{GL}_{k}(\mathbb{Z}/p).$
Therefore, $g^{\prime }$ induces a map $g^{\prime \prime }:\mathrm{PSL}_{n}(%
\mathbb{Z}/N)\rightarrow \mathrm{PGL}_{k}(\mathbb{Z}/p).$ However, it's
known that $g^{\prime \prime }$ has to be trivial by Landazuri--Seitz \cite%
{ls}.
\end{proof}

Let $A$ be a finite abelian group. For a prime $p,$ define the $p$-rank $%
\mathrm{rank}_{p}(A)$ as the dimension of $A\bigotimes_{\mathbb{Z}/p}\mathbb{%
Z}/p,$ as a vector space over $\mathbb{Z}/p.$

\begin{lemma}
\label{ab}Let $A$ be a finite abelian group with $\mathrm{rank}_{p}(A)<n-1$
for every prime $p.$ Then every group homomorphism $\mathrm{St}_{n}(\mathbb{Z%
})\rightarrow \mathrm{Aut}(A)$ is trivial for $n\geq 3$.
\end{lemma}

\begin{proof}
Since a group homomorphism preserves $p$-Sylow subgroup of $A,$ it's enough
to prove the theorem for a $p$-group $A.$ If $A$ is an elementary $p$-group, 
$\mathrm{Aut}(A)=\mathrm{GL}_{k}(\mathbb{Z}/p),$ $k=\mathrm{rank}_{p}(A).$
Thus $\mathrm{Hom}(\mathrm{St}_{n}(\mathbb{Z}),\mathrm{Aut}(A))$ is trivial
by Lemma \ref{trg}. If $A$ is not elementary, the subgroup $A_{1}$
consisting of elements of order $p$ is a characteristic subgroup $A$.
Inductively, we assume that $\mathrm{Hom}(\mathrm{St}_{n}(\mathbb{Z}),%
\mathrm{Aut}(A/A_{1}))$ and $\mathrm{Hom}(\mathrm{St}_{n}(\mathbb{Z}),%
\mathrm{Aut}(A_{1}))$ are both trivial. Let 
\begin{equation*}
B=\{f\in \mathrm{Aut}(A):f|_{A_{1}}=\mathrm{id}_{A_{1}}\}.
\end{equation*}%
The group $\ker (\mathrm{Aut}(A)\rightarrow \mathrm{Aut}(A/A_{1}))\cap
B=H^{1}(A/A_{1};A_{1})$ is abelian (cf. \cite{gu}, Prop. 5, p.45). Since $%
\mathrm{St}_{n}(\mathbb{Z})$ $(n\geq 3)$ is perfect, we have $\mathrm{Hom}(%
\mathrm{St}_{n}(\mathbb{Z}),\mathrm{Aut}(A))$ is trivial.
\end{proof}

\begin{lemma}
\label{snil}Let $\pi $ be a finite nilpotent group with the upper central
series $1=Z_{0}<Z_{1}<\cdots <Z_{k}=\pi .$ Suppose that $\mathrm{rank}%
_{p}(Z_{i}/Z_{i-1})<n-1$ for each prime $p$ and each $i=1,\cdots ,k.$ Then
the set of group homomorphisms $\mathrm{Hom}(\mathrm{St}_{n}(\mathbb{Z}),%
\mathrm{Out}(\pi ))=1.$
\end{lemma}

\begin{proof}
When $k=1,$ this is proved in Lemma \ref{ab}. Considering the central
extension%
\begin{equation*}
1\rightarrow Z_{i}/Z_{i-1}\rightarrow \pi /Z_{i-1}\rightarrow \pi
/Z_{i}\rightarrow 1
\end{equation*}%
the statement could be proved inductively using Lemma \ref{ab} and Lemma \ref%
{ni}.
\end{proof}

\begin{corollary}
\label{app1}Let $M^{r}$ be a closed almost flat manifold with the holonomy
group $\Phi $ nilpotent satisfying the condition in Lemma \ref{snil}. Then
Conjecture \ref{conj} holds for $M.$ In particular, when $M$ is a closed
flat manifold with abelian holonomy group, Conjecture \ref{conj} holds.
\end{corollary}

\begin{proof}
By Lemma \ref{snil}, $\mathrm{Hom}(\mathrm{St}_{n}(\mathbb{Z}),\mathrm{Out}%
(\Phi ))=1$ and thus $\mathrm{Hom}(\mathrm{SL}_{n}(\mathbb{Z}),\mathrm{Out}%
(\Phi ))=1.$ Theorem \ref{th3} implies that Conjecture \ref{conj} is true.
When $M$ is flat, $\Phi $ is a subgroup of $\mathrm{GL}_{r}(\mathbb{Z}).$ If 
$\Phi $ is abelian, elements in $\Phi $ could be simultaneously
diagonalizable in $\mathrm{GL}_{r}(\mathbb{C}).$ Therefore, $\mathrm{rank}%
_{p}(\Phi )\leq r-1<n-1$ for each prime $p.$
\end{proof}

\bigskip

\begin{lemma}
\label{di}Let $\Phi $ be a dihedral group $D_{2k}$, symmetric group $S_{k}$
or alternating group $A_{k}.$ Then $\mathrm{Hom}(\mathrm{SL}_{n}(\mathbb{Z}),%
\mathrm{Out}(\Phi ))=1$ for $n\geq 3.$
\end{lemma}

\begin{proof}
It's well-known that $\mathrm{Aut}(D_{2k})=\mathbb{Z}/k\rtimes \mathrm{Aut}(%
\mathbb{Z}/k),$ a solvable group. Therefore, $\mathrm{Out}(D_{2k})$ is
solvable. When $\Phi $ is $S_{k}$ or $A_{k},$ we have $\mathrm{Out}(\Phi )$
is abelian. However, $\mathrm{SL}_{n}(\mathbb{Z})$ is perfect and thus $%
\mathrm{Hom}(\mathrm{SL}_{n}(\mathbb{Z}),\mathrm{Out}(\Phi ))=1.$
\end{proof}

For flat manifolds of low dimensions, we obtain the following.

\begin{corollary}
\label{lflat}Conjecture \ref{conj} holds for closed flat manifolds $M^{r}$of
dimension $r\leq 5.$
\end{corollary}

\begin{proof}
The proof depends on the classifications of low-dimensional holonomy groups $%
\Phi $. When $r\leq 3,$ $\Phi =\{1\},\mathbb{Z}/2,\mathbb{Z}/3,\mathbb{Z}/4,%
\mathbb{Z}/6$ or $\mathbb{Z}/2\times \mathbb{Z}/2$ (cf. \cite{wo} Corollary
3.5.6, p.118). They are all abelian. By Lemma \ref{app1}, Conjecture \ref%
{conj} is true. When $r=4,$ the nonabelian $\Phi =D_{6},D_{8},D_{12}$ or $%
A_{4}$ (cf. \cite{bro}). By Lemma \ref{di} and Theorem \ref{th3}, Conjecture %
\ref{conj} is true. When $r=5,$ the nonabelian $\Phi
=D_{6},D_{8},D_{12},D_{8}\times \mathbb{Z}/2,D_{6}\times \mathbb{Z}%
/3,D_{12}\times \mathbb{Z}/2,A_{4},A_{4}\times \mathbb{Z}/2,A_{4}\times 
\mathbb{Z}/2\times \mathbb{Z}/2,S_{4}$ or $(\mathbb{Z}/2\times \mathbb{Z}%
/2)\rtimes \mathbb{Z}/4$ (cf. \cite{sz}, Theorem 1, or \cite{cs}, Theorem
4.2). By Lemma \ref{di}, it's enough to consider the $\Phi $ with two
factors. By Lemma \ref{snil} and Lemma \ref{ni}, $\mathrm{Hom}(\mathrm{SL}%
_{n}(\mathbb{Z}),\mathrm{Out}(\Phi ))=1$ and thus Conjecture \ref{conj}
holds by Theorem \ref{th3}.
\end{proof}

Department of Mathematical Sciences, Xi'an Jiaotong-Liverpool University,
111 Ren Ai Road, Suzhou, Jiangsu 215123, China.

E-mail: Shengkui.Ye@xjtlu.edu.cn

\end{document}